\documentclass[12pt]{amsart}
\usepackage{amssymb}
\usepackage[all]{xy}
\newtheorem{theorem}{Theorem}[]
\newtheorem{lemma}[theorem]{Lemma}

\newtheorem{definition}[theorem]{Definition}
\newtheorem{example}[theorem]{Example}
\newtheorem{remark}[theorem]{Remark}
\def\to{\rightarrow}
\def\f{\mathfrak}
\def\m{\mathbb}
\def\r{\mathrm}
\def\c{\mathcal}
\def\b{\mathbf}

\def\i{\iota}
\def\ot{\otimes}
\begin{document}
\setcounter{page}{1}
\title[Additive Equation on Quantum Semigroups]{Stability of Additive Functional Equation on\\
Discrete Quantum Semigroups}
\author[M. M. Sadr]{Maysam Maysami Sadr}
\address{Department of Mathematics\\
Institute for Advanced Studies in Basic Sciences\\
P.O.Box 45195-1159, Zanjan 45137-66731, Iran}
\email{sadr@iasbs.ac.ir}
\subjclass[2010]{Primary 81R15, Secondary 81R60, 39B42, 34D99}
\keywords{Discrete quantum semigroup, Additive functional equation, Hyers-Ulam stability, Noncommutative geometry}
\begin{abstract}
We show that noncommutative analog of additive functional equation has Hyers-Ulam stability on amenable discrete quantum
(semi)groups. This generalizes an old classical result.
\end{abstract}
\maketitle
\section{Introduction}
Let $G$ be a (semi)group. Consider the additive functional equation (\textbf{AFE}),
$$F(xy)=F(x)+F(y),$$
for functions $F$ from $G$ to the complex field $\m{C}$. \textbf{AFE} is said to have Hyers-Ulam stability (\textbf{HUS})
on $G$ if the following property holds.
\begin{quote}
\emph{Given $r>0$, there is $r'>0$ such that if a function $f$ on $G$ satisfies
$$|f(xy)-f(x)-f(y)|<r'$$
then there exists a function $F$ on $G$ satisfying
$$F(xy)=F(x)+F(y)\hspace{5mm}\text{and}\hspace{5mm}|F(x)-f(x)|<r.$$}
\end{quote}
The study of the above property goes back to a famous question of Ulam \cite{Ulam1} for characterization of pairs $(G,H)$,
where $H$ is a metric group, satisfying the above property with $\m{C}$ replaced by $H$. In 1941, Hyers \cite{Hyers1} showed that
if $G$ is the underlying additive group of a Banach space then  \textbf{AFE} has \textbf{HUS} on $G$.
Four decades later, Forti \cite{Forti1} extended
the result of Hyers for amenable semigroups by a very simple method. Since the appearance of \cite{Hyers1}, the Ulam stability
problem and its generalizations not only for \textbf{AFE} but also other types of functional equations has been considered and
developed by many mathematicians. (See \cite{Jung1} for the history of developments.) Nowadays, this area of mathematics is
generally named Hyers-Ulam stability.

The main goal of this note that we wish it would be the first one of a series of papers is an invitation to
stability theory of functional equations on noncommutative spaces. We start our program to study this subject by considering
the same traditional problem of Ulam for quantum groups instead of ordinary groups. Indeed, we extend the above mentioned result of
Forti \cite{Forti1} as follows. (For exact definitions of discrete quantum semigroups and amenability see Section 3.) Let $\m{G}$
be a discrete quantum semigroup with comultiplication $\Delta$. Denote by $\b{F}(\m{G})$ the function algebra on $\m{G}$
and by $\b{F}_\r{b}(\m{G})$ the von Neumann subalgebra of bounded functions on $\m{G}$. The 'sup-norm' of $\b{F}_\r{b}(\m{G})$ is denoted
by $\|\cdot\|$. The noncommutative analog of \textbf{AFE} becomes
$$\Delta(F)=1\ot F+F\ot1,$$
for functions $F$ in $\b{F}(\m{G})$. Similar the above mentioned stability property we make
\begin{definition}
We say that noncommutative \textbf{AFE} has \textbf{HUS} on $\m{G}$ if the following condition holds.
For every $r>0$ there is $r'>0$ such that if a function $f\in\b{F}(\m{G})$ satisfies the inequality
$$\|\Delta(f)-1\ot f-f\ot 1\|<r',$$
then there exists a function $F\in\b{F}(\m{G})$ for which
$$\Delta(F)=1\ot F+F\ot 1\hspace{10mm}\text{and}\hspace{10mm}\|F-f\|<r.$$
\end{definition}
The main result of this note is the following theorem that may be considered as an extension of \cite[Theorem 7]{Forti1}
to discrete quantum semigroups.
\begin{theorem}
If $\m{G}$ is a left or right amenable discrete quantum semigroup then noncommutative \textbf{AFE} has \textbf{HUS} on $\m{G}$.
\end{theorem}
Our proof of Theorem 2 that will be given in Section 4 is the same proof of Forti \cite[Theorem 7]{Forti1} but translated to the dual
language of Hopf-algebras. As it would  be clear for the reader by taking a quick look at Forti's proof, for this
dualization we need to work with unbounded 'functions' of two variables which are bounded by fixing one of the variables.
Moreover, we must have a machinery to apply bounded operators to spaces of such functions. In Section 2
following some ideas from \cite{EffrosRuan1} we introduce somewhat a new way for tensoring linear maps which is specially 
designated to overcome the mentioned difficulties.
In Section 3 we consider definition of discrete quantum semigroups and amenability. Our definition is the same one of Van Daele
\cite{Daele2} for discrete quantum groups but with weaker conditions which result semigroups. Also these can be considered as
Hopf-von Neumann algebras of discrete type \cite{EnockSchwartz1} expect that their coproducts need not to be injective.
We end this section with three remarks.
\begin{remark}
\begin{enumerate}
\item[(i)] Suppose that $F\in\b{F}(\m{G})$ as above satisfies in noncommutative \textbf{AFE}.
Since $F$ as a function takes values in finite dimensional matrix algebras we may consider $E=\exp(F)$ as a member of $\b{F}(\m{G})$.
Then it is straightforward to check that $\Delta(E)=E\ot E$. Such elements in the language of Hopf-algebras are called group-like.
\item[(ii)] One can consider Definition 2 for any locally compact quantum group $\m{G}$ \cite{KustermansVaes1} where $f$ and $F$ are
unbounded affiliated operators to the underlying von Neumann algebra of $\m{G}$. But the proof of Theorem 2
does not longer work in this more general case.
\item[(iii)] Some works on stability of noncommutative analog of quadratic, Jensen and $n$-difference functional equations on quantum groups
and Kac algebras are in process.
\end{enumerate}
\end{remark}
\section{A type of tensor products}
Throughout $\i$ denotes the identity map and the C*-algebra of $n\times n$ matrices is denoted by $\c{M}_{n}$.
By an index system we mean a set $I$ together with a positive integer valued function $n_I$ on $I$. If $\gamma\in I$ then
for simplicity we write $\c{M}_\gamma$ for $\c{M}_{n_I(\gamma)}$. In the following $I,I',J,J'$
denote index systems. We denote by $\b{F}(I)$ the *-algebra of all functions
$f:I\to \cup_{\gamma\in I}\c{M}_\gamma$ for which $f(\gamma)\in\c{M}_\gamma$, with pointwise operations.
(In \cite{EffrosRuan1} $\b{F}(I)$ is called multimatrix algebra.) The *-subalgebra of all functions
with finite support is denoted by $\b{F}_\r{f}(I)$. So in the standard notation $\b{F}(I)=\prod_{\gamma\in I}\c{M}_\gamma$ and
$\b{F}_\r{f}(I)=\bigoplus_{\gamma\in I}\c{M}_\gamma$. It is also simply verified that $\b{F}(I)$ is identified with the multiplier algebra
of $\b{F}_\r{f}(I)$. We denote the unit of $\b{F}(I)$ by $1$ and hence $1(\gamma)=1_\gamma$ is the identity matrix in $\c{M}_\gamma$.
$\b{F}_\r{b}(I)$ is the *-subalgebra of $\b{F}(I)$ containing bounded functions i.e. those functions $f$ for which
$\|f\|=\sup_{\alpha\in I}\|f(\alpha)\|<\infty$. This is a C*-algebra with the $\sup$-norm and is the dual space of absolutely sumable functions.
So $\b{F}_\r{b}(I)$ is a von Neumann algebra. Let $I_i$ be an index system for $i=1,\ldots,k$.
We consider the cartesian product set $I_1\times\cdots\times I_k$ as an index system with
$n_{I_1\times\cdots\times I_k}(\alpha_1,\cdots,\alpha_k)=n_1(\alpha_1)\cdots n_k(\alpha_k)$.
Let $A=\{i_1,\cdots,i_l\}$ be a subset of $\{1,\ldots,k\}$. Then we let $\b{F}_{b:i_1\cdots i_l}(I_1\times\cdots\times I_k)$ be the subspace
of those functions $f$ in $\b{F}(I_1\times\cdots\times I_k)$ such that
for every fixed family $\{\alpha_i\in I_i\}_{i\in\{1,\ldots,k\}\setminus A}$, the condition
$\sup_{\alpha_i\in I_i,i\in A}\|f(\alpha_1,\cdots,\alpha_k)\|<\infty$ holds.

Suppose that $T$ is a linear map from $\b{F}(I )$ (resp. $\b{F}_\r{b}(I )$)
to $\b{F}(I')$. We define a linear map $T\tilde\ot\i$ from $\b{F}(I \times J)$ (resp. $\b{F}_{\r{b}:1}(I \times J)$) to $\b{F}(I'\times J)$
as follows. For $\beta\in J$ let $\i_\beta$ denote the identity linear map on $\c{M}_\beta$. Let $f$ be in $\b{F}(I \times J)$
(resp. $\b{F}_{\r{b}:1}(I \times J)$). Since $\c{M}_\beta$ is finite dimensional the function $\alpha \mapsto f(\alpha ,\beta)$ determines
a unique member of $\b{F}(I )\ot\c{M}_\beta$ (resp. $\b{F}_\r{b}(I )\ot\c{M}_\beta$). So $(T\ot\i_\beta)(\alpha \mapsto f(\alpha ,\beta))$
is in $\b{F}(I')\ot\c{M}_\beta$. Considering this latter space as a space of functions from $I'$ to
$\cup_{\alpha'\in I'}\c{M}_{\alpha'}\ot\c{M}_\beta$ we let
$[(T\tilde\ot\i)(f)](\alpha',\beta)=[(T\ot\i_\beta)(\alpha \mapsto f(\alpha ,\beta))](\alpha')$.
We may also write a more explicit formula for $T\tilde\ot\i$ as follows. Let $\{e_\beta^{ij}\}_{1\leq i,j\leq n_J(\beta)}$
be the standard vector basis for $\c{M}_\beta$. For $f$ as above let the elements $f_\beta^{ij}$ of $\b{F}(I )$ (resp. $\b{F}_\r{b}(I )$)
be such that $f(\alpha ,\beta)=\sum_{ij}f_\beta^{ij}(\alpha )\ot e_\beta^{ij}$. Then
\begin{equation}\label{E1}
[(T\tilde\ot\i)(f)](\alpha',\beta)=\sum_{ij}[T(f_\beta^{ij})](\alpha')\ot e_\beta^{ij}.
\end{equation}
We remark that if $T$ is a linear functional then the image of $T\tilde\ot\i$ canonically belongs to $\b{F}(J)$.
We may define similarly linear maps $\i\tilde\ot T$ and $\i\tilde\ot T\tilde\ot\i$. So, the latter is a map from
$\b{F}(J\times I\times J')$ (resp. $\b{F}_{\r{b}:2}(J\times I\times J')$) to $\b{F}(J\times I'\times J')$.
In below we list some properties of $\tilde\ot$ which are used in next sections.
\begin{enumerate}
\item[(\textbf{P0})] \emph{If $T$ is a linear map from $\b{F}(I)$ (resp. $\b{F}_\r{b}(I)$) to $\b{F}(I')$ then
$$(T\tilde\ot\i)(f\ot g)=T(f)\ot g,$$
where $f$ is in $\b{F}(I)$ (resp. $\b{F}_\r{b}(I)$) and $g\in\b{F}(J)$, and $f\ot g$ denotes the function
$(\alpha,\beta)\mapsto f(\alpha)\ot g(\beta)$.}
\end{enumerate}
Another trivial property of $\tilde\ot$ is associativity:
\begin{enumerate}
\item[(\textbf{P1})] \emph{$(\i\tilde\ot T)\tilde\ot\i=\i\tilde\ot T\tilde\ot\i=\i\tilde\ot (T\tilde\ot\i)$ and
$(T\tilde\ot\i)\tilde\ot\i=T\tilde\ot\i\tilde\ot\i=T\tilde\ot(\i\tilde\ot\i)$.}
\end{enumerate}
From (\ref{E1}) it follows easily that:
\begin{enumerate}
\item[(\textbf{P2})] \emph{If $T$ is a linear map from $\b{F}(I)$ or $\b{F}_\r{b}(I)$ to $\b{F}_\r{b}(I')$ then the image of $T\tilde\ot\i$
is contained in $\b{F}_{\r{b}:1}(I'\times J)$. The analogous statements are satisfied for  $\i\tilde\ot T$ and $\i\tilde\ot T\tilde\ot\i$.}
\end{enumerate}
For $\alpha\in I$ and $\beta\in J$ let $\f{P}_{\beta}:\b{F}(J)\to\c{M}_{\beta}$ and $\f{I}_{\alpha}:\c{M}_{\alpha}\to \b{F}(I)$
denote canonical linear projection and imbedding respectively. For every linear map $T:\b{F}(I)\to\b{F}(J)$ we let
$T^\alpha_\beta=\f{P}_{\beta}T\f{I}_{\alpha}$. Now, suppose that $T$ is a *-homomorphism from $\b{F}(I)$
to $\b{F}(J)$. Since the kernel of $\f{P}_{\beta}T$ is a two-sided ideal with finite codimension in $\b{F}(I)$,
and since matrix algebras have no nontrivial two-sided ideals, there is a finite subset $I_0$ of $I$ with $\f{P}_{\beta}T|_{\b{F}(I\setminus I_0)}=0$.
It follows that for every fixed $\beta\in J$ there are only finitely many $\alpha$ in $I$ with $T^\alpha_\beta\neq0$ and
$[T(f)](\beta)=\sum_\alpha T^\alpha_\beta(f(\alpha))$. Analogous statements are completely satisfied when the domain of $T$
is the subalgebra $\b{F}_\r{b}(I)$.
\begin{enumerate}
\item[(\textbf{P3})] \emph{If $T$ is a *-homomorphism from $\b{F}(I)$ or $\b{F}_\r{b}(I)$ to $\b{F}(J)$ then
$$[(T\tilde\ot\i)(f)](\beta,\beta')=\sum_{\alpha}(T^{\alpha}_{\beta}\ot\i_{\beta'})(f(\alpha,\beta'))\hspace{5mm}(\beta'\in J').$$
The analogous statements are satisfied for  $\i\tilde\ot T$ and $\i\tilde\ot T\tilde\ot\i$.}
\end{enumerate}
\begin{enumerate}
\item[(\textbf{P4})] \emph{Let $T:\b{F}(I)\to\b{F}(J)$ and $T':\b{F}(I')\to\b{F}(J')$ be linear maps such that either $T$ or $T'$ is
*-homomorphism. Then
$$(\i\tilde\ot T')(T\tilde\ot\i)=(T\tilde\ot\i)(\i\tilde\ot T')$$
as linear maps from $\b{F}(I\times I')$ to $\b{F}(J\times J')$.}
\end{enumerate}
\begin{proof}
We suppose that $T$ is *-homomorphism. The other case is similar. Let $f$ be in $\b{F}(I\times I')$ and let
$f^{ij}_\alpha\in\b{F}(I')$ ($1\leq i,j\leq n_I(\alpha)$) be such that $f(\alpha,\alpha')=\sum_{ij}e^{ij}_\alpha\ot f^{ij}_\alpha(\alpha')$.
By (\textbf{P3}),
$[(T\tilde\ot\i)(f)](\beta,\alpha')
=\sum_\alpha\sum_{ij}T^\alpha_\beta(e^{ij}_\alpha)\ot f^{ij}_\alpha(\alpha')$.
This implies that
$[(\i\tilde\ot T')(T\tilde\ot\i)(f)](\beta,\beta')=\sum_\alpha\sum_{ij}T^\alpha_\beta(e^{ij}_\alpha)\ot [T'(f^{ij}_\alpha)](\beta')$.
On the other hand $[(\i\tilde\ot T')(f)](\alpha,\beta')=\sum_{ij}e^{ij}_\alpha\ot [T'(f^{ij}_\alpha)](\beta')$ and hence
\begin{align*}
[(T\tilde\ot\i)(\i\tilde\ot T')(f)](\beta,\beta')&=
\sum_{\alpha}(T^\alpha_\beta\ot\i_{\beta'})(\sum_{ij}e^{ij}_\alpha\ot [T'(f^{ij}_\alpha)](\beta'))\\
&=\sum_\alpha\sum_{ij}T^\alpha_\beta(e^{ij}_\alpha)\ot [T'(f^{ij}_\alpha)](\beta').
\end{align*}
\end{proof}
\begin{enumerate}
\item[(\textbf{P5})] \emph{Suppose that $T:\b{F}(I)\to\b{F}(J)$ is a *-homomorphism. If $f$ belongs to
$\b{F}_{\r{b}:2}(I\times J')$ then $(T\tilde\ot\i)(f)\in\b{F}_{\r{b}:2}(J\times J')$.}
\end{enumerate}
\begin{proof}
Let $f\in\b{F}_{\r{b}:2}(I\times J')$. So for every $\alpha\in I$ we have $\sup_{\beta'\in J'}\|f(\alpha,\beta')\|<\infty$.
Let $\beta\in J$ be fixed. By (\textbf{P3}), $[(T\tilde\ot\i)(f)](\beta,\beta')=\sum_{\alpha}(T^{\alpha}_{\beta}\ot\i_{\beta'})(f(\alpha,\beta'))$.
So $\sup_{\beta'\in J'}\|[(T\tilde\ot\i)(f)](\beta,\beta')\|
\leq\sum_{\alpha\in I,T^\alpha_\beta\neq0}(\sup_{\beta'\in J'}\|f(\alpha,\beta')\|)<\infty$.
\end{proof}
\section{Discrete quantum semigroups}
Let $I$ be an index system. A comultiplication for $I$ is a collection of
*-homomorphisms $\Delta^\alpha_{\beta,\gamma}:\c{M}_\alpha\to \c{M}_\beta\ot\c{M}_\gamma$ for each ordered triple
$(\alpha,\beta,\gamma)$ of elements of $I$, which satisfies the two conditions below.
\begin{enumerate}
\item[(i)] $\Delta^\alpha_{\beta,\gamma}(1_\alpha)\Delta^{\alpha'}_{\beta,\gamma}(1_{\alpha'})=0$ for $\alpha\neq\alpha'$, and
\item[(ii)] the *-homomorphisms
$\sum_{\omega}(\Delta^\omega_{\alpha,\beta}\ot\i)\Delta^\lambda_{\omega,\gamma}$ and
$\sum_{\omega}(\i\ot\Delta^\omega_{\beta,\gamma})\Delta^\lambda_{\alpha,\omega}$ from $\c{M}_\lambda$ to
$\c{M}_\alpha\ot\c{M}_\beta\ot\c{M}_\gamma$ are equal.
\end{enumerate}
Note that (i) implies that for fixed $\beta$ and $\gamma$ there are only finitely many $\alpha$ with $\Delta^\alpha_{\beta,\gamma}(1_\alpha)\neq0$.
Also (i) guarantees that the *-homomorphisms in (ii) are well defined. Now, we may, and hence do, define a *-homomorphism $\Delta$ by
$$[\Delta(f)](\beta,\gamma)=\sum_{\alpha}\Delta^\alpha_{\beta,\gamma}f(\alpha),$$
from $\b{F}(I)$ to $\b{F}(I\times I)$. Then (ii) may be restated as
$(\Delta\tilde\ot\i)\Delta=(\i\tilde\ot\Delta)\Delta$. Note that every $\Delta^\alpha_{\beta,\gamma}$ may be recovered from $\Delta$.
So, from now on we do not distinguish between $\Delta$ and the collection $\{\Delta^\alpha_{\beta,\gamma}\}$.
\begin{definition}
A discrete quantum semigroup is a pair $\m{G}=(I,\Delta)$ such that $I$ is an index system and $\Delta$ is a comultiplication for $I$.
\end{definition}
For a discrete quantum semigroup $\m{G}=(I,\Delta)$ we denote the algebras $\b{F}_\r{b}(I)$ and $\b{F}(I)$, respectively,
by $\b{F}_\r{b}(\m{G})$ and $\b{F}(\m{G})$. Analogously, we let $\b{F}(\m{G}\times\m{G})=\b{F}(I\times I)$ and
$\b{F}_\r{b}(\m{G}\times\m{G})=\b{F}_\r{b}(I\times I)$. The comultiplication $\Delta$ of $\m{G}$ transforms bounded functions
to bounded functions i.e. $\Delta(\b{F}_\r{b}(\m{G}))\subseteq\b{F}_\r{b}(\m{G}\times\m{G})$. It follows from the fact that
the map $f\mapsto[\Delta(f)](\beta,\gamma)$ from $\b{F}_\r{b}(\m{G})$ to $\c{M}_\beta\ot\c{M}_\gamma$ is a *-homomorphism between
C*-algebras and hence norm decreasing.

Let $\m{G}=(I,\Delta)$ be a discrete quantum semigroup. Then $\m{G}$ is called right amenable \cite{BedosMurphyTuset1} if there is a
state $\f{m}$ on $\b{F}_\r{b}(\m{G})$, called right invariant mean, which satisfies $(\f{m}\tilde\ot\i)\Delta(f)=\f{m}(f)1$
for every $f\in\b{F}_\r{b}(\m{G})$. Left invariant means and left amenable discrete quantum semigroups are defined similarly.
\begin{example}
Let $G$ be a discrete semigroup. Then $G$ gives rise to a discrete quantum semigroup $\m{G}=(I,\Delta)$
in which $I=G$ and $n_I=1$. The *-homomorphisms $\Delta^\alpha_{\beta,\gamma}:\m{C}\to\m{C}\ot\m{C}=\m{C}$ are defined by
$$\Delta^\alpha_{\beta,\gamma}=\begin{cases}
\i & \alpha=\beta\gamma\\
0& \text{otherwise}
\end{cases}$$
In this case, $\m{G}$ is right (resp. left) amenable iff $G$ is right (resp. left) amenable as usual. Also, it is not so hard to see
that every discrete quantum semigroup $\m{G}=(I,\Delta)$ for which $n_I=1$, is constructed from a discrete semigroup, as above.
\end{example}
Discrete quantum groups \cite{Daele2} which are Pontryagin dual of compact quantum groups \cite{MaesDaele1}
(or Hopf-von Neumann algebras of discrete type \cite{EnockSchwartz1}) are also discrete quantum semigroups in our sense.
We will need the next lemmas in Section 4.
\begin{lemma}
Let $\m{G}=(I,\Delta)$ be a discrete quantum semigroup and $\f{m}$ be a right invariant mean for $\m{G}$. Then for every
$f\in\b{F}_{\r{b}:1}(\m{G}\times\m{G})$ the following holds.
$$(\f{m}\tilde\ot\i\tilde\ot\i)(\Delta\tilde\ot\i)(f)=1\ot[(\f{m}\tilde\ot\i)(f)].$$
\end{lemma}
\begin{proof}
First of all, note that by (\textbf{P2}) the right hand side is well-defined. Let $f$ be in
$\b{F}_{\r{b}:1}(\m{G}\times\m{G})$ and let $f^{ij}_\gamma\in\b{F}_{\r{b}}(\m{G})$ be such that
$f(\omega,\gamma)=\sum_{ij}f^{ij}_\gamma(\omega)\ot e^{ij}_\gamma$. Then we get
$[(\Delta\tilde\ot\i)(f)](\alpha,\beta,\gamma)=\sum_{ij}[\Delta(f^{ij}_\gamma)](\alpha,\beta)\ot e^{ij}_\gamma$ and hence
\begin{align*}
[(\f{m}\tilde\ot\i\tilde\ot\i)(\Delta\tilde\ot\i)(f)](\beta,\gamma)&=[((\f{m}\tilde\ot\i)\tilde\ot\i)(\Delta\tilde\ot\i)(f)](\beta,\gamma)\\
&=\sum_{ij}[(\f{m}\tilde\ot\i)\Delta(f^{ij}_\gamma)](\beta)\ot e^{ij}_\gamma\\
&=\sum_{ij}\f{m}(f^{ij}_\gamma)1_\beta\ot e^{ij}_\gamma\\
&=\sum_{ij}1_\beta\ot \f{m}(f^{ij}_\gamma)e^{ij}_\gamma\\
&=1_\beta\ot[(\f{m}\tilde\ot\i)(f)](\gamma)\\
&=(1\ot[(\f{m}\tilde\ot\i)(f)])(\beta,\gamma).
\end{align*}
\end{proof}
\begin{lemma}
Let $\m{G}=(I,\Delta)$ be a discrete quantum semigroup, $\f{n}$ be a linear functional on $\b{F}_\r{b}(\m{G})$, and
$f\in\b{F}_{\r{b}:1}(\m{G}\times\m{G})$. Then
$$\Delta(\f{n}\tilde\ot\i)(f)=(\f{n}\tilde\ot\i\tilde\ot\i)(\i\tilde\ot\Delta)(f).$$
\end{lemma}
\begin{proof}
First of all, note that by (\textbf{P5}) the right hand side is well-defined. Let $\bar{\f{n}}$ be an arbitrary linear extension of $\f{n}$
to $\b{F}(\m{G})$. Then
\begin{align*}
(\f{n}\tilde\ot\i\tilde\ot\i)(\i\tilde\ot\Delta)(f)&=(\bar{\f{n}}\tilde\ot\i\tilde\ot\i)(\i\tilde\ot\Delta)(f)\\
&=\Delta(\bar{\f{n}}\tilde\ot\i)(f)=\Delta(\f{n}\tilde\ot\i)(f),
\end{align*}
where we have used (\textbf{P4}) to pass from the first equality to the second one.
\end{proof}
\section{The proof of Theorem 2}
Suppose that $\m{G}$ is right amenable. The proof of the other case is similar. Let $\f{m}$ be a right invariant mean for $\m{G}$ and
let $r>0$ be given. We show that the conditions of Definition 1 are satisfied for $r'=r$. Let $f\in\b{F}(\m{G})$ be such that
$$\|\Delta(f)-1\ot f-f\ot 1\|<r,$$
that is $\sup_{\beta,\gamma}\|[\Delta(f)](\beta,\gamma)-f(\beta)\ot1_\gamma-1_\beta\ot f(\gamma)\|<r$. It follows that
$$\sup_{\beta}\|[\Delta(f)](\beta,\gamma)-f(\beta)\ot1_\gamma\|<r+\|1_\beta\ot f(\gamma)\|=r+\|f(\gamma)\|.$$
So $(\Delta(f)-f\ot 1)\in\b{F}_{\r{b}:1}(\m{G}\times\m{G})$. We define a function $F$ in $\b{F}(\m{G})$ by
$$F=(\f{m}\tilde\ot\i)(\Delta f-f\ot1).$$
By (\textbf{P0}) and (\textbf{P1}) we get
\begin{equation}\label{EA1}
F\ot1=[(\f{m}\tilde\ot\i)(\Delta(f)-f\ot1)]\ot\i(1)\\
=(\f{m}\tilde\ot\i\tilde\ot\i)(\Delta (f)\ot1-f\ot1\ot1).
\end{equation}
It follows from Lemma 6 that $1\ot F=(\f{m}\tilde\ot\i\tilde\ot\i)(\Delta\tilde\ot\i)(\Delta(f)-f\ot1)$ and hence
by (\textbf{P2}) we get
\begin{equation}\label{EA2}
1\ot F=(\f{m}\tilde\ot\i\tilde\ot\i)((\Delta\tilde\ot\i)\Delta(f)-\Delta(f)\ot1).
\end{equation}
It follows from (\ref{EA1}) and (\ref{EA2}) that
\begin{align*}
F\ot1+1\ot F&=(\f{m}\tilde\ot\i\tilde\ot\i)((\Delta\tilde\ot\i)\Delta(f)-f\ot1\ot1)\\
&=(\f{m}\tilde\ot\i\tilde\ot\i)((\i\tilde\ot\Delta)\Delta(f)-f\ot1\ot1)\\
&=(\f{m}\tilde\ot\i\tilde\ot\i)(\i\tilde\ot\Delta)(\Delta(f)-f\ot1),
\end{align*}
where we have used (\textbf{P5}) to pass from the second row to the third one.
By Lemma 7 the third row is equal to $\Delta(\f{m}\tilde\ot\i)(\Delta f-f\ot1)=\Delta(F)$. So
we shaw that $\Delta(F)=1\ot F+F\ot1$. For the norm inequality we have
\begin{align*}
\|F-f\|&=\|(\f{m}\tilde\ot\i)(\Delta(f)-f\ot1)-f\|\\
&=\|(\f{m}\tilde\ot\i)(\Delta(f)-f\ot1)-(\f{m}\tilde\ot\i)(1\ot f)\|\\
&=\|(\f{m}\tilde\ot\i)(\Delta(f)-f\ot1-1\ot f)\|\\
&\leq\|(\f{m}\tilde\ot\i)\|\|\Delta(f)-f\ot1-1\ot f\|<r.
\end{align*}
This completes the proof.
\bibliographystyle{amsplain}

\begin{thebibliography}{99}
\bibitem{BedosMurphyTuset1} E. B\`{e}dos, G.J. Murphy, L. Tuset, \emph{Amenability and co-amenability of algebraic quantum groups},
Int. J. Math. Math. Sci. \textbf{31} (2002) 577–-601.
\bibitem{EnockSchwartz1} M. Enock, J.-M. Schwartz, \emph{Kac algebras and duality of locally compact
groups}, Springer-Verlag, Berlin--Heidelberg--New York, 1992.
\bibitem{EffrosRuan1} E.G. Effros, Z.-J. Ruan, \emph{Discrete Quantum Groups I, The Haar Measure}, Int. J. Math.
\textbf{5} (1994) 681--723.
\bibitem{Forti1} G.L. Forti, \emph{The stability of homomorphisms and amenability, with applications to functional
equations}, Abh. Math. Sem. Univ. Hamburg \textbf{57} (1987) 215–-226.
\bibitem{Hyers1} D.H. Hyers, \emph{On the stability of the linear functional equation}, Proc. Nat. Acad. Sci. U.S.A. \textbf{27}
(1941) 222–-224.
\bibitem{Jung1} S.-M. Jung, \emph{Hyers–-Ulam–-Rassias stability of functional equations in nonlinear analysis}, Springer, 2011.
\bibitem{KustermansVaes1} J. Kustermans, S. Vaes, \emph{Locally compact quantum groups}, Ann. Sci. \'{E}cole Norm. Sup.
(4) \textbf{33} (2000) 837–-934.
\bibitem{MaesDaele1} A. Maes, A. Van Daele, \emph{Notes on compact quantum groups}, Nieuw Arch. Wisk.
\textbf{16} (1998) 73–-112.
\bibitem{Ulam1} S.M. Ulam, \emph{A collection of the mathematical problems}, Interscience Publ., New York, 1960.
\bibitem{Daele2} A. Van Daele, \emph{Discrete quantum groups}, J. Algebra \textbf{180} (1996) 431--444.
\end{thebibliography}

\end{document}